\documentclass{amsart}
\usepackage{a4wide, amsmath, amssymb, epsfig, graphicx, enumerate, booktabs}

\let\oldmarginpar\marginpar
\renewcommand\marginpar[1]{\-\oldmarginpar[\raggedleft\footnotesize #1]%
{\raggedright\footnotesize #1}}

\renewcommand{\P}{\mathcal{P}}
\newcommand{\Q}{\mathcal{Q}}

\renewcommand{\leq}{\leqslant}
\renewcommand{\geq}{\geqslant}

\newcommand{\gen}[1]{\bigl<#1\bigr>}
\newcommand{\sym}[1]{S{(#1)}}

\newcommand{\id}{{\rm id}}
\newcommand{\im}{\operatorname{im}}
\newcommand{\rank}{\mbox{\emph{\mbox{\rm rank}\,}}}

\newcommand{\genset}[1]{\langle#1\rangle}
\renewcommand{\to}{\longrightarrow}

\newcommand{\F}{\mathbb F}
\newtheorem{theorem}{Theorem}[section]

\newtheorem{lemma}[theorem]{Lemma}

\newtheorem{cor}[theorem]{Corollary}
\newtheorem{exam}[theorem]{Example}
\newtheorem{defn}[theorem]{Definition}
\newtheorem{prob}[theorem]{Problem}

\title[The rank of the endomorphisms of a partition]
{The rank of the semigroup of transformations stabilising a partition
of a finite set}
\author{Jo\~ao Ara\'ujo, Wolfram Bentz, J. D. Mitchell \and Csaba Schneider}
\address[Ara\'ujo]{Universidade Aberta, R. Escola Polit\'{e}cnica, 147\\
1269-001 Lisboa, Portugal\\
Centro de \'{A}lgebra, Universidade de Lisboa\\
1649-003 Lisboa, Portugal, jaraujo@lmc.fc.ul.pt}
\address[Bentz]{Centro de \'{A}lgebra, Universidade de Lisboa\\
1649-003 Lisboa, Portugal, wfbentz@fc.ul.pt}
\address[Schneider]{Departamento de Matem\'tica \\
Instituto de Ci\^encias Exatas \\
Universidade Federal de Minas Gerais \\
Av. Ant\^onio Carlos, 6627 \\
Caixa Postal 702 \\
31270-901 Belo Horizonte, MG, Brazil, csaba@mat.ufmg.br}

\date{\today}

\begin{document}
\begin{abstract}
Let $\mathcal{P}$ be a partition of a finite set $X$. We say that a full transformation $f:X\to X$ preserves (or stabilizes) the partition $\mathcal{P}$ if for all $P\in \mathcal{P}$ there exists $Q\in \mathcal{P}$ such that $Pf\subseteq Q$. Let  $T(X,\mathcal{P})$ denote the semigroup of all full transformations of $X$ that preserve the partition $\mathcal{P}$.

In 2005 Huisheng found an upper bound for the minimum size of the generating sets of $T(X,\mathcal{P})$, when  $\mathcal{P}$ is a partition in which all of its parts have the same size. In addition, Huisheng conjectured that his bound was exact. In 2009 the first and last authors used representation theory to completely solve Hisheng's conjecture.

The goal of this paper is to solve the much more complex problem of finding the minimum size of the generating sets of $T(X,\mathcal{P})$, when  $\mathcal{P}$ is an arbitrary partition. Again we use representation theory to find the minimum number of elements needed to generate the wreath product of finitely many symmetric groups, and then use this result to solve the problem.

The paper ends with a number of problems for experts in group and semigroup theories.
\end{abstract}
\maketitle

\section{Introduction}

If $S$ is a semigroup and $U$ is a subset of $S$, then we say that {\em $U$
generates $S$} if every element of $S$ is expressible as a product of the
elements of $U$. The {\em rank} of a semigroup $S$, denoted by $\rank S$, is the
least cardinality of a subset that generates $S$.  It is well-known that a
finite full transformation semigroup, on at least 3 points, has rank~3, while a
finite full partial transformation semigroup, on at least 3 points, has rank~4
(see \cite[Exercises~1.9.7 and 1.9.13]{howie}). The problem of determining the
minimum number of generators of a semigroup is classical, and has been studied
extensively; see, for example, \cite{1,8,10,15,16}. Related notions, such as the
{\em idempotent rank}, the {\em nilpotent rank} and the relative rank of a
subsemigroup, have also been widely investigated; see
\cite{0,2,mitchell1,6,9,mitchell2,8.5,11,14}.

In~\cite{huisheng}, Huisheng posed the problem of finding the rank of the
semigroup of transformations preserving a uniform partition (that is, a
partition in which all the blocks have equal size). This problem was solved in
\cite{as}.  In this paper, we solve the general problem of determining the rank
of the semigroup of transformations preserving any partition.  In the process,
we calculate the ranks of some related transformation semigroups.  The strategy
of the proof is similar to the one used in \cite{as}: we rely on representation
theory to find the rank of the group of automorphisms of the partition and then
use that result to derive the rank of the semigroup.

Let $X$ be a non-empty finite set, and let $\P$ be a partition of $X$.
A \emph{transformation} is a function from $X$ to itself. We write
transformations to the right of their arguments and compose them from left to
right.  We denote by $T(X, \P)$ the semigroup consisting of those
transformations $f$ on $X$ such that $(x,y)\in \P$ implies $(xf, yf)\in
\P$.  The semigroup $T(X, \P)$ can be seen as the endomorphism
monoid of the relational structure $(X,\P)$.

We will determine the rank of $T(X,\P)$. In order to do this we will
determine relative ranks with regard to two subsets  of $T(X, \P)$. One
is the group of units of $T(X, \P)$, which is the intersection of $T(X,
\P)$ with the symmetric group $S_X$ on $X$; the other is $\Sigma(X,
\P)$, consisting of $f \in T(X, \P)$ whose image intersects
every block of $\P$. We will denote the group of units of $T(X,
\P)$ by $S(X, \P)$.

The main theorem of this paper is the following.

\begin{theorem}
  Let $\P$ be a partition on $X$, such that $\P$ has exactly $m_i \ge 2$ blocks
  of size $n_i \ge 2$, $i=1, \dots, p$, blocks of unique sizes
  $l_1, \dots , l_q$, where $l_i \ge 2$, and $t$ singleton blocks (where
  $p$, $q$, $t$ might be $0$).  If $|S(X, \P)| \ge 3$  then the rank of $T(X, \P)$ is
  $$\max\{2,2p+q+g(t)\}+\binom{p+q}{2}+2p+q+g'(t)-1+l + h(p,q,t),$$
  where
\begin{itemize}
\item $g(0)=g(1)=0$ and $g(t)=1$ for $t \ge 2$,
\item  $g'(0)=0$ and $g'(t)=1$ for $t \ge 1$.
\item $l$ is the number of values $s$ for which $\P$ has a block of size $s\ge
  2$, but no block of size $s-1$,
\item $h(p,q,0)=0$, $h(p,q,1)=p+q$ and $h(p,q,t)=p+q+1$, if $t \ge 2$.
\end{itemize}
\end{theorem}
The rank of $T(X,\P)$ is given in Figure \ref{ranktable} for partitions of small values of
$|X|$. For comparison, Figure \ref{monoidtable} lists the corresponding sizes 
of the monoids $T(X,\P)$.

If $U$ is a subset of a semigroup $V$, then, as usual we denote the subsemigroup
generated by $U$ by $\genset{U}$.  If $U$ is a subsemigroup of a semigroup $V$,
then the least cardinality of a subset $W$ of $V$ such that $\genset{U, W}=V$ is
called the \emph{relative rank} of $U$ in $V$; this is denoted $\rank(V:U)$.

Since $S(X, \P)\subseteq \Sigma(X, \P)$ and the complements of $S(X, \P)$ and
$\Sigma(X, \P)$ are ideals, it follows that
\begin{equation*}
  \begin{array}{rcl}
    \rank(T(X,\P))&=&\rank(T(X,\P):\Sigma(X,\P))+\rank(\Sigma(X,\P)) \\
                  &=&\rank(T(X,\P):\Sigma(X,\P))+\rank(\Sigma(X,\P):S(X,
    \P))+\rank(S(X, \P))
   \end{array}
\end{equation*}

To prove our main theorem, we will determine that under the given conditions
\begin{itemize}
  \item $\rank(S(X, \P))= \max\{2,2p+q+g(t)\}$
 (Section \ref{section-group-of-units}),
\item $\rank(T(X,\P):\Sigma(X,\P))= \binom{p+q}{2}+p+h(p,q,t)$ (Section \ref{section-rel-rank-1}), and
\item $\rank(\Sigma(X,\P):S(X, \P))=p+q+g'(t)-1+l$ (Section \ref{section-rel-rank-2}).
\end{itemize}
For completeness, we remark that if $S(X,\P)$ has two elements, we are in one of the following straightforward cases:
\begin{itemize}
\item $|X|=2$, $\P=\{P_1\}$, $|P_1|=2$, $\rank(T(X,\P))=2$.
\item $|X|=2$, $\P=\{P_1, P_2\}$, $|P_1|=|P_2|=1$, $\rank(T(X,\P))=2$.
\item $|X|=3$, $\P=\{P_1, P_2\}$, $|P_1|=2$, $|P_2|=1$, $\rank(T(X,\P))=3$.
\end{itemize}
\begin{figure}\label{ranktable}
  \begin{center}
    \begin{tabular}{lrclrclrclrclr}
      2+1 & 3 && 2+1+1 & 5 && 2+1+1+1 & 5 && 2+1+1+1+1 & 5 && 2+1+1+1+1+1 & 5
          \\\midrule
          &   && 2+2   & 4 && 2+2+1   & 5 && 2+2+1+1   & 7 && 2+2+1+1+1   & 7
          \\\cmidrule{4-14}
          &   && 3+1   & 5 && 3+1+1   & 6 && 2+2+2     & 4 && 2+2+2+1     & 5
          \\\cmidrule{4-14}
          &   &&       &   && 3+2     & 5 && 3+1+1+1   & 6 && 3+1+1+1+1   & 6
          \\\cmidrule{7-14}
          &   &&       &   && 4+1     & 5 && 3+2+1     & 7 && 3+2+1+1     & 9
          \\\cmidrule{7-14}
          &   &&       &   &&         &   && 3+3       & 4 && 3+2+2       & 7
          \\\cmidrule{10-14}
          &   &&       &   &&         &   && 4+1+1     & 6 && 3+3+1       & 6
          \\\cmidrule{10-14}
          &   &&       &   &&         &   && 4+2       & 6 && 4+1+1+1     & 6
          \\\cmidrule{10-14}
          &   &&       &   &&         &   && 5+1       & 5 && 4+2+1       & 8
          \\\cmidrule{10-14}
          &   &&       &   &&         &   &&           &   && 4+3         & 5
          \\\cmidrule{13-14}
          &   &&       &   &&         &   &&           &   && 5+1+1       & 6
          \\\cmidrule{13-14}
          &   &&       &   &&         &   &&           &   && 5+2         & 6
          \\\cmidrule{13-14}
          &   &&       &   &&         &   &&           &   && 6+1         & 5
          \\\cmidrule{13-14}
 \end{tabular}
    \caption{The partitions of 3 to 7 and the ranks of the
    corresponding monoids.}
  \end{center}
\end{figure}
\begin{figure}
  \begin{center}
    \begin{tabular}{lrclrclrclrclr}
2+1 & 6 && 2+1+1 & 96  && 2+1+1+1 & 875  && 2+1+1+1+1 & 10368 && 2+1+1+1+1+1 & 151263
    \\\midrule
    &   && 2+2   & 64  && 2+2+1   & 405  && 2+2+1+1   & 3600  && 2+2+1+1+1   & 41503
    \\\cmidrule{4-14}
    &   && 3+1   & 100 && 3+1+1   & 725  && 2+2+2     & 1728  && 2+2+2+1     & 15379
    \\\cmidrule{4-14}
    &   &&       &     && 3+2     & 455  && 3+1+1+1   & 6480  && 3+1+1+1+1   & 74431
    \\\cmidrule{7-14}
    &   &&       &     && 4+1     & 1285 && 3+2+1     & 3024  && 3+2+1+1     & 27195
    \\\cmidrule{7-14}
    &   &&       &     &&         &      && 3+3       & 2916  && 3+2+2       & 12427
    \\\cmidrule{10-14}
    &   &&       &     &&         &      && 4+1+1     & 9288  && 3+3+1       & 21175
    \\\cmidrule{10-14}
    &   &&       &     &&         &      && 4+2       & 5440  && 4+1+1+1     & 88837
    \\\cmidrule{10-14}
    &   &&       &     &&         &      && 5+1       & 18756 && 4+2+1       & 40131
    \\\cmidrule{10-14}
    &   &&       &     &&         &      &&           &       && 4+3         & 30667
    \\\cmidrule{13-14}
    &   &&       &     &&         &      &&           &       && 5+1+1       & 153223
    \\\cmidrule{13-14}
    &   &&       &     &&         &      &&           &       && 5+2         & 91553
    \\\cmidrule{13-14}
    &   &&       &     &&         &      &&           &       && 6+1         & 326599
    \\\cmidrule{13-14}
 \end{tabular}
    \caption{The partitions of 3 to 7 and the sizes of the
    corresponding monoids.}\label{monoidtable}
  \end{center}
\end{figure}


\section{The rank of direct products of wreath products of symmetric groups}
\label{section-group-of-units}

If $G$ and $H$ are permutation groups,
then we denote by $G\wr H$ the \emph{wreath product} of $G$ and $H$.
As usual, if $|X|=n$, then we denote the symmetric group $S_X$ on $X$ by $S_n$;
likewise, in this case, we denote the alternating group by $A_n$.

Let $n \ge 2$ and let $\P$ be a partition with at least $2$ parts.
Then we may write $\P=\{P_1,\ldots, P_n\}$ such
that $|P_i|\le |P_j|$ when $i<j$, and $i,j\in \{1,\ldots ,n\}$.

If $f\in T(X,\P)$, then we denote by $\overline{f}\in T_n$ the
transformation whose action on $\{1,2,\ldots,n\}$ is that induced by the action
of $f$ on $X/\P$. In more details, $(i)\overline{f}=j$ whenever
$P_if\subseteq P_j$.  If $f\in S(X, \P)$, then it is clear that
$(i)\overline{f}=j$ if and only if $|P_i|=|P_j|$.

We start by stating without proof two simple results about $S(X, \P)$
and its induced action on $T(X, \P)$.

\begin{lemma}
  Let $\P$ be a partition of a set $X$ where the distinct sizes of the
  blocks are denoted $n_i$, $i=1,\ldots, k$, and $m_i$ denotes the number of
  blocks of size $n_i$.  Then the group of units $S(X,\P)$ of $T(X, \P)$ is
  isomorphic to $$(S_{n_1}\wr S_{m_1})\times\cdots\times (S_{n_k}\wr S_{m_k}).$$ \end{lemma}

If $f$ is a transformation of a set $X$, then the \emph{image} of $f$
is the set $$\im(f)=\{(x)f:x\in X\}$$ and the \emph{kernel} of $f$
is the equivalence relation $$\ker(f)=\{(x,y)\in X\times X:(x)f=(y)f\},$$
the classes of this relation are referred to as \emph{kernel classes}.
If $Y$ is a subset of $X$, then the \emph{restriction} of $f$ to $Y$ is denoted
by $f|_Y$.

\begin{lemma} \label{lem:Gaction}

For every block $P$ of $\P$ and $f \in T(X, \P)$, let $P_{f}$ be the multiset
of sizes of blocks in the kernel of $f|P_i$. For every $i,j$ such that $\P$ has
blocks of sizes $i$ and $j$ (not necessarily distinct), let $J_{i,j,f}$ be the
multiset of all $P_f$ such that $|P|=i$ and $(P)f$ is contained in a block of size
$j$. Then $g \in S(X,\P) f S(X,\P)$ if and only if $J_{i,j,f}=J_{i,j,g}$ for all pairs
$(i,j)$.
\end{lemma}
For example, let $X=\{1, \dots,8\}$, $\P=\{P,Q\}$, $P=\{1,2,3,4\}, Q=\{5,6,7,8\}$,
and define $f \in T(X,\P)$ by $(1)f=2$, $(3)f=4$, and $(x)f=x$ for
$x \ne 1,3$. Then $P_f=\{2,2\}$, $Q_f=\{1,1,1,1\}$, and $J_{4,4,f}= \{ \{2,2\},
\{1,1,1,1\}\} $.

If $g \in T(X,\P)$ is given by $(1)g=2$, $(5)g=6$, and $(x)g=x$ for $x \ne 1,5$,
then $P_g=\{2,1,1\}=Q_g$, and $J_{4,4,g}= \{ \{2,1,1\}, \{2,1,1\}\} $. Hence $g
\not\in S(X,\P) f S(X,\P)$. Note that $f$ and $g$ have the same multiset of
sizes of kernel classes.

We recall also one of the main theorems in \cite{as}.

\begin{theorem}\label{wrth}
  If $X$ is a finite set such that $|X|\geq 3$ and $\P$ is a uniform partition
  of $X$, then $\sym{X,\P}$ is generated by two elements.
\end{theorem}

The following lemma is well-known; see, for instance,~\cite[Lemma~5.3.4]{kl}.

\begin{lemma}\label{replem}
  The permutation module $V$ of the symmetric group $S_n$ on an $n$-element set
  over a field $\F$ of characteristic $p$ has precisely two proper non-trivial
  submodules:
  \begin{eqnarray*}
    U_1&=&\left\{(a,a,\ldots,a)\ :\ a\in \F\right\}\quad\mbox{and}\\
    U_2&=&\left\{(a_1,\ldots,a_n)\ :\ a_1+\cdots+a_n=0 \right\}.
  \end{eqnarray*}
  Furthermore, if $p$ divides $n$, then $U_1\leq U_2$; otherwise $V=U_1\oplus
  U_2$.
\end{lemma}

\begin{theorem}\label{th:SXP}
  Let $n_1,\ldots,n_k, m_1,\ldots,m_k ,l_1,\ldots,l_u$ be integers such that they are all at least $2$ and let
  $$W=(S_{n_1}\wr S_{m_1})\times\cdots\times (S_{n_k}\wr S_{m_k})\times
  S_{l_1}\times\cdots\times S_{l_u}.$$
  If $W\not\cong S_2$, then the rank of $W$ is $\max\{2,\ 2k+u\}$.
\end{theorem}

\begin{proof}
  Let us assume that $W\not\cong C_2$. If $2k+u<2$, then $k=0$ and $u=1$. In
  this case, $W=S_{l_1}$ is not isomorphic to $S_2$, and the rank of $W$ is 2.

  Let us show that $W$ cannot be generated by fewer than $2k+u$ elements.  Let
  $i\in\{1,\ldots,k\}$. Then $(A_{n_i})^{m_i}$ is a normal subgroup of
  $S_{n_i}\wr S_{m_i}$ and the quotient $Q$ is isomorphic to $C_2\wr
  S_{m_i}=(C_2)^{m_i}\rtimes S_{m_i}$.  Then $(C_2)^{m_i}$ can be viewed as the
  natural permutation module for $S_{m_i}$ over $\F_2$. If $U_2$ denotes the
  $S_{m_i}$-submodule
    of $(C_2)^{m_i}$ defined in Lemma~\ref{replem}, then $U_2$
  is a normal subgroup of $Q$ and the quotient is isomorphic to $C_2\times
  S_{m_i}$. Now $A_{m_i}$ is a normal subgroup of $C_2\times S_{m_i}$ and the
  quotient is isomorphic to $C_2\times C_2$. Therefore we have proved that the
  wreath product $S_{n_i}\wr S_{m_i}$ has a normal subgroup $N_i$ such that the
  quotient is isomorphic to $C_2\times C_2$. Now, for $i\in\{1,\ldots,u\}$, the
  subgroup $A_{l_i}$ normal in $S_{l_i}$ and the quotient is isomorphic to
  $C_2$. Therefore the subgroup $$ N=N_1\times\cdots\times N_k\times
  A_{l_1}\times\cdots\times A_{l_u} $$ is a normal subgroup of $W$ such that
  $W/N$ is isomorphic to $(C_2)^{2k+u}$.  If $W$ can be generated by fewer than
  $2k+u$ elements, then so can $W/N\cong (C_2)^{2k+u}$. However, the smallest
  generating set of $(C_2)^{2k+u}$ has $2k+u$ elements, and so the assertion is
  verified.

  Next we show that $W$ can indeed be generated by $2k+u$ elements.
  Set
  $$
  W_1=(S_{n_1}\wr S_{m_1})\times\cdots\times (S_{n_k}\wr
  S_{m_k})\quad\mbox{and}\quad W_2=S_{l_1}\times\cdots\times S_{l_u}.
  $$
  Then $W=W_1\times W_2$.
  Since $W_1$ is the direct product of $k$ groups each of which is generated by
  two elements (Theorem~\ref{wrth}), we obtain that $W_1$
  can be generated by $2k$ elements. For $u=0$ the theorem is thus proved.

  Suppose that $u=1$. If $k=0$ then $S_{l_1}$ can be generated by $2$ elements
  and there is nothing to prove. Suppose that $k\geq 1$ and consider the group
  $H=(S_{n_k}\wr S_{m_k})\times S_{l_1}$. By the argument in the previous
  paragraph, it suffices to show that $H$ is generated by $3$ elements. Let $x$
  and $y$ be the generators of $S_{n_k}\wr S_{m_k}$ given in Theorem~\ref{wrth}.
  Set $u=(x,\id)$, $v=(y,(1,2,\ldots,l_1))$, and $w=(\id,(1,2))$. Then $u,\ v,\
  w\in H$ and we claim that $H=\gen{u,v,w}$. Since the first components of $u,\
  v,\ w$ generate $S_{n_k}\wr S_{m_k}$ and the second components generate
  $S_{l_1}$, we have that $\gen{u,v,w}$ is a subdirect subgroup of
  $H=(S_{n_k}\wr S_{m_k})\times S_{l_1}$. Set $N=\gen{u,v,w}\cap S_{l_1}$. For
  each $u_2\in  S_{l_1}$ there is some $u_1\in S_{n_k}\wr S_{m_k}$ such that
  $(u_1,u_2)\in \gen{u,v,w}$. If $n\in N$ then
  $(\id,n)^{(u_1,u_2)}=(\id,n^{u_2})$ is an element of $N=\gen{u,v,w}\cap
  S_{l_1}$, and this shows that $N$ is a normal subgroup of $S_{l_1}$. As
  $(1,2)\in N$ and no proper normal subgroup of $S_{l_1}$ contains the
  transposition $(1,2)$, we find that $N=S_{l_1}$, and, in turn, that
  $S_{l_1}\leq \gen{u,v,w}$. As $\gen{u,v,w}$ is subdirect, we also obtain
  $S_{n_k}\wr S_{m_k}\leq \gen{u,v,w}$, and so $H=\gen{u,v,w}$. Thus shows that
  $H$ is generated by three elements, and so $W$ is generated by $2k+1$
  elements, as required.

  Suppose now that $u\geq 2$. In this case, as $W_1$ is generated by $2k$
  elements, we only need to show that $W_2$ is generated by $u$ elements.
  Let $i\in\{1,\ldots,u\}$. If $l_i$ is even, then set $z_i=(2,\ldots,l_i)$ otherwise
  set $z_i=(1,\ldots,l_i)$. Therefore $z_i$ is always a cycle of odd length such
  that $S_{l_i}=\gen{(1,2),z_i}$.  For $i\in\{1,\ldots,u-1\}$ define
  $$
    w_i=(\id,\ldots,\id,\stackrel{\mbox{\footnotesize $i$-th
        component}}{(1,2)},\stackrel{\mbox{\footnotesize $(i+1)$-th
        component}}{z_{i+1}},\id,\ldots,\id)\in W_2
  $$
  and also define
  $$
    w_u=(z_1,\id,\ldots,\id,(1,2))\in W_2.
  $$
  We claim that $W_2=\gen{w_1,\ldots,w_u}$.
  Let $o_i$ denote the order of $z_i$.
  As $o_i$ is odd, all but the $i$-th component of
  $w_i^{o_i}$ is trivial, and the $i$-th component is $(1,2)$. If
  $i\in\{1,\ldots,u-1\}$, then in $w_i^{2}$, all but
  the $(i+1)$-th component is trivial, and the $(i+1)$-th component is
  $z_{i+1}^2$.
  Similarly, in $w_u^{2}$ all but the first component is
  trivial, and the first component is $z_1^2$. Therefore,
  for $i\in\{1,\ldots,u\}$, we  obtain that the elements $(1,2),\ z_i^2\in
  S_{l_i}$ are contained in $\gen{w_1,\ldots,w_u}$. Since, the order of $z_i$ is
  odd, $z_i^2$ is a cycle of the same length as $z_i$ permuting the same
  points. Therefore $\gen{(1,2),z_i^2}=S_{l_i}$, which shows that $S_{l_i}\leq
  \gen{w_1,\ldots,w_u}$. Since this is true for all $i$, we obtain that $W_2\leq
  \gen{w_1,\ldots,w_u}$,  and the proof is complete.
\end{proof}
\begin{cor}
Let $\P$ be a partition on $X$, such that $\P$ has exactly $m_i \ge 2$ blocks
  of size $n_i \ge 2$, $i=1, \dots, p$, blocks of unique sizes
  $l_1, \dots , l_q$, where $l_i \ge 2$, and $t$ singleton blocks (where
  $p$, $q$, $t$ might be $0$).  If $|S(X, \P)| \ge 3$  then the rank of $S(X, \P)$ is
  $$\max\{2,2p+q+g(t)\},$$
where $g(0)=g(1)=0$ and $g(t)=1$ for $t \ge 2$.
\end{cor}
\begin{proof}
If $t=0$ or $t=1$ then $S(X,\P)$ is isomorphic to
$$(S_{n_1}\wr S_{m_1})\times\cdots\times (S_{n_p}\wr S_{m_p})\times
  S_{l_1}\times\cdots\times S_{l_q}$$
and we may take $k=p$ and $u=q$ in Theorem \ref{th:SXP}. If $t\ge 2$ then $S(X,\P)$ is isomorphic
to
$$(S_{n_1}\wr S_{m_1})\times\cdots\times (S_{n_p}\wr S_{m_p})\times
  S_{l_1}\times\cdots\times S_{l_q} \times S_t,$$
and the statement follows from Theorem \ref{th:SXP} with $k=p$, $u=q+1$.
\end{proof}

\section{The Relative rank of $T(X,\P)$ modulo $\Sigma(X,\P)$}
\label{section-rel-rank-1}

Let $\mathcal{A}$ denote the collection of those $f\in T(X, \P)$ such that
\begin{enumerate}[(i)]
  \item $f|_{P_i}$ is injective for all $i$;
  \item $|\im(\overline{f})|=n-1$;
\end{enumerate}
Note that, by (i), if $(P_i)f\subseteq P_j$ and $|P_i|\not=|P_j|$, then
$|P_i|<|P_j|$.

\begin{lemma} \label{lemma-rel-rank}
  Let $f, g, a\in T(X, \P)$ be arbitrary. Then the following hold:
  \begin{enumerate}[\rm (i)]
    \item if $f\in \mathcal{A}$, $a\in \Sigma(X, \P)$, and $f=ag$, then $a\in
      S(X, \P)$;

       \item if $f\in \mathcal{A}$, $g\not\in \Sigma(X, \P)$, and $f=ga$, then
      $g\in \mathcal{A}$.

    \item if $f, g\in \mathcal{A}$ and $\overline{f}=\overline{g}\overline{a}$, then there
      exist unique $i,j\in \{1,\ldots, n\}$  such that $i\not=j$ and
      $(i)\overline{f}=(j)\overline{f}=(i)\overline{g}=(j)\overline{g}$;

  \end{enumerate}
\end{lemma}
\begin{proof}
  {\bf (i).}
  Since $\ker(a)\subseteq \ker(f)$ and $f$ is injective on every $P_i\in \P$, it
  follows that $a$ is injective on every $P_i\in \P$. But $a\in \Sigma(X, \P)$
  and so $\overline{a}$ is a permutation. Thus $a$ is a permutation, i.e.
  $a\in S(X, \P)$.

  {\bf (ii).} As in the previous case, $\ker(g)\subseteq \ker(f)$, and since
  $f$ is injective on every part of $\P$, it follows that $g$ is too.
  Since $g\not\in \Sigma(X,\P)$,  $|\im(\overline{g})|\le n-1$.  If
  $|\im(\overline{g})|<n-1$, then $|\im(\overline{ga})|<n-1=|\im(\overline{f})|$, a
  contradiction. So $|\im(\overline{g})|=n-1$, and $g \in \mathcal{A}$.

  {\bf (iii).}  Similar to the previous cases, $\overline{f}=\overline{g}\overline{a}$ implies
   that $\ker(\overline{g})\subseteq \ker(\overline{f})$. But $f, g \in
  \mathcal{A}$, which implies that $|\im(f)|=|\im(g)|=n-1$ and hence
  $\ker(\overline{f})=\ker(\overline{g})$.

\end{proof}

\begin{lemma}\label{lemma-jdm-1}
  Let $U\subseteq T(X,\P)\setminus\Sigma(X,\P)$ be such that $T(X,\P)=\langle
  \Sigma(X,\P), U\rangle$. Then for all distinct $i,j\in\{1,\ldots, n\}$ there exist
  $f\in U\cap \mathcal{A}$ and distinct $k, l\in \{1,\ldots, n\}$ such that
  $(k)\overline{f}=(l)\overline{f}$ and $|P_i|=|P_k|$ and $|P_j|=|P_l|$.
\end{lemma}
\begin{proof}
  Let $i,j\in \{1,\ldots, n\}$ be arbitrary. Then there exists $f\in
  \mathcal{A}$ such that $(i)\overline{f}=(j)\overline{f}$.
  By assumption,  $f\in \langle \Sigma(X,\P),U\rangle$ and so
  $$f=s_1a_1s_2a_2\ldots s_ra_ns_{r+1}\quad\text{for some}\quad s_i\in
  \Sigma(X,\P),\ a_i\in U.$$
  By Lemma \ref{lemma-rel-rank}(i), $s_1\in S(X, \P)$ and so $s_1^{-1}f\in
  \mathcal{A}$. If $k=(i)\overline{s_1}$ and $l=(j)\overline{s_1}$, then
  $(k)\overline{s_1^{-1}f}=(l)\overline{s_1^{-1}f}$. Since $s_1\in S(X,
  \P)$, it follows that $|P_k|=|P_i|$ and $|P_l|=|P_j|$.
  By Lemma \ref{lemma-rel-rank}(iii), $a_1\in \mathcal{A}$ and thus, by Lemma
  \ref{lemma-rel-rank}(ii), $(k)\overline{a_1}=(l)\overline{a_1}$, as required.
\end{proof}

We have everything we need to prove the main result of this
section.

\begin{theorem}\label{th:TXP:SigXP}

  Let $X$ be any finite set, $\P$ a partition of $X$,
  $s$ the number of distinct values $|P_i|$ and $r$ the number of
  distinct values $|P_i|$ such that there are $i \ne j$ with $|P_i| =
  |P_j|$. Then $\rank(T(X,\P):\Sigma(X,\P))=\binom{s}{2}+r$.
\end{theorem}
\begin{proof}
By Lemma \ref{lemma-jdm-1}, $\rank(T(X,\P):\Sigma(X,\P))\geq \binom{s}{2}+r$.

For the converse direction,
given positive integers $p \le q$, let
$\mathcal{A}_{p,q}$ be the set of all $f \in \mathcal{A}$
such that $|P_i|=p, |P_j|=q$ where $i$, $j$ are the members of the unique
non-singleton class of $\ker \bar f$. It is clear that the non-empty $\mathcal{A}_{p,q}$ form a partition of
$\mathcal{A} $ with $\binom{s}{2}+r$ parts.

Let $U^{\P}$ be a set of representatives $f_{p,q}$ for the elements of this partition.
We
claim that $U^{\P}$ generates $T(X,\P)$ over $\Sigma(X, \P)$. We will first
show that a particular set of functions can be generated from $U^{\P} \cup
\Sigma(X,\P)$.

For each $i<j$, and $\phi$ an injection from $P_i$ to $P_j$, let $f_{i,j, \phi}
\in T(X, \P)$ be the function that agrees with $\phi$ on $P_i$ and is the
identity everywhere else. By Lemma \ref{lem:Gaction}, $f_{i,j,\phi} \in
S(X,\P)f_{|P_i|,|P_j|}S(X,\P)$, and hence in $\langle \Sigma(X,\P) \cup
U^{\P}\rangle$.

Let $f \in T(X,\P)$.  We will show that $f$ can be generated from
$\Sigma(X,\P)$ and the $f_{i,j, \phi}$ using induction on the number of blocks in a partition $\Q$ of an
arbitrary finite set $X'$. The base case when there is only one block is
trivial, since in this case $\Sigma(X',\Q)=T(X',\Q)$.
Our induction assumption is that $T(X', \Q)$ is generated by $U^{\Q}$ and
$\Sigma(X',\Q)$ when $Y$ is any finite set and $\Q$ is any partition of $Y$
with fewer than $n\in \mathbb{N}$, $n>1$, blocks.

We construct several
functions that are generated by $\Sigma(X,\P)$ and the $f_{i,j,\phi}$ until we
are able to use our inductive hypothesis.

Let $l=(n) \bar f$, and $i_1, \dots, i_k=n$ be the elements of the $\bar
f$-kernel class of $n$. Let $D=P_{i_1} \cup \dots \cup P_{i_k}$.

Choose an injective function $h$ from  $P_l$ to $ P_n$.  This is possible as
$|P_n| \ge |P_l|$. Moreover for all $j$ choose an injective function $h_j$ from
$\im(f|_{P_{i_j}})$  to $P_{i_j}$. Such $h_l$ exist, as the image
$\im(f|_{P_{i_j}})$ is not larger than the domain $P_{i_j}$.

Let $e$ be the function for which $e|_D$ maps $y \in P_{i_j}$ to $((y)f) h_j$
and $e|_{X \setminus D}$ is the identity.  Then $\bar e$ is the identity, and
hence $e \in \Sigma(X, \P)$.

For each $j$, let $\phi_j$ be a function from $P_{i_j}$ to $P_n$ defined in
the following way.  For  $x\in \im(h_j)$, by construction there exists a $y \in
P_{i_j}$ such that $x=((y)f)h_j$. In this case set $(x)\phi_j= ((y)f)h$.

We
have to show that this definition does not depend on the choice of $y$. So let
$((y_1)f)h_j = ((y_2)f)h_j$ for some $y_1, y_2\in P_{i_j}$.
  As $h_j$ is an injection defined on the image of
$f|_{P_{i_j}}$, we have that $(y_1)f=(y_2)f$, and hence $((y_1)f)h=((y_2)f)h$
and so $\phi_j$ is well-defined for every $x \in \im(h_j)$.

Now let
$x_1 \ne x_2$, $x_1, x_2 \in \im(h_j)$, say $((y_1)f)h_j=x_1$
and $((y_2)f)h_j=x_2$. Then $(y_1)f \ne (y_2)f$ and as $h$ is injective
$(x_1) \phi_j = ((y_1)f)h \ne ((y_2)f)h=(x_2) \phi_j$. Hence $\phi_j$ is injective on
$\im(h_j)$.
Now extend $\phi_j$ arbitrary to all of  $P_{i_j}$, subject
to $\phi_j$ being an injection. Such $\phi_j$ exists as $P_n$ is the block of
largest size.  Let $$g = e f_{i_1,n, \phi_1}  \dots f_{i_{k},n,\phi_{k}}.$$

It is straightforward to check to  that $g$ satisfies the following properties:
\begin{enumerate} \item $g|_{X \setminus D}$ is the identity, \item $
    (\{i_1,\dots,i_k\}) \bar g=\{n\}$, and $\bar g$ is the identity otherwise,
  \item  $g|_D$ and $f|_D$ have the same kernel, \item if $x \in D$, then $(x)
g= ((x)f)h$.  \end{enumerate}

Next let $v$ be a function constructed  as follows:  $v$ maps $P_n$ to $P_l$ so
that any $x \in \im(h)$ is mapped to $x h^{-1}$ and is arbitrary otherwise
(recall that $h$ was injective).  For $j=l, \dots, n-1$, $v$ maps $P_j$
injectively into $P_{j+1}$ and is the identity everywhere else. Clearly such $v$
exists and is an element of $\Sigma(X, \rho)$. Let $g'=gv$. We claim that $g'$
has the following properties:

\begin{enumerate}
  \item $g'|_{X \setminus D}$ is injective,
  \item $ (\{i_1,\dots,i_k\}) \bar g'=\{l\}$, and $\bar g'$ maps all other
    values injectively to values different from $l$,
  \item if $x \in D$, then $(x) g'= (x)f$,
  \item $\ker g'  \subseteq \ker f$,
  \item if $(x)g'$ and $(y)g'$ are in the same part of $\P$, then $(x)f$ and $(y)f$
    are in the same part of $\P$.
\end{enumerate}
The first two properties follow from the corresponding results for  $g$. For the
third, let $x \in D$,  then $(x) g'=((x)g)v= (((x)f)h)v=(x)f$, and so $g'$
agrees with $f$ on $D$.  The fourth assertion follows from the first and third,
and the final one from the second and fourth one.

 We will next construct a new function $h'$. If $x \in \im(g')$, say $x = (y)
 g'$, then set $(x)h'=(y)f$. As $\ker g'  \subseteq \ker f$, this function is
 well-defined. By the last property of $g'$, this partial assignment preserves $\P$.
If $x \not\in \im(g')$, $x \in P_i$ and there is a $y \in P_i \cap \im(g')$,
then let $(x)h'=x$ if $(y)h' \in P_i$, or otherwise
be an arbitrary element of the part of $(y)h'$. Once again by
the last property of $g'$, the condition is well defined and the assignment so
far continues to preserve $\P$.  Finally if $i$ is such that $P_i \cap \im(g')$
is empty, then pick a $P_j \in \P$, with $j \ne l$, and let $h'$ map all of
$P_i$ into $P_j$ in an arbitrary way.

The function $h'$ has the following properties:
 \begin{enumerate}
 \item $h' \in T(X, \P)$
 \item  $g'h'=f$
 \item $(\{l\}) \bar h'^{-1}=(\{l\})$.
 \item $h'$ is the identity on $P_l$.
\end{enumerate}
 The first two properties and the fact that $(P_l) h' \subseteq P_l$ follow
 directly from the construction of $h'$. Conversely let $x \in P_i$ with $i \ne
 l$. If $\im (g') \cap P_i$ is empty, then the above construction maps $x$ into
 a part different from $P_l$. If there is an element in $\im(g') \cap P_i$,
 which we may assume w.l.o.g. to be $x$, let $x=(y)g'$. Then $y \not \in
 \{P_{i_1},\dots,P_{i_k}\}$ by the second property of $g'$. But then
 $(x)h'=(y)f$ cannot be in $P_l$ as $(\{l\})\bar f^{-1}=\{i_1,\dots,i_k\}$. So
 $(\{l\}) \bar h'^{-1}=(\{l\})$.

 For the last property, let $x \in P_l \cap \im(g')$, say $(y)g'=x$. By property (2)
 of $g'$, $x \in D$, and hence, by property (3) of $g'$, $(x)h'=(y)f=(y)g'=x$. As $h'$
 maps the elements of $P_l \cap \im(g')$ into $P_l$, it maps
 $P_l \setminus \im(g')$ identical by its definition.

Now let $X' = X\setminus P_l$ and $\Q$ be the partition of $X'$ given by $\P
\setminus \{P_l\}$.  Let $f'=h'|_{X\setminus P_l}$. As $(\{l\}) \bar
h'^{-1}=(\{l\})$, $f' \in T(X', \Q)$. By the induction assumption, $f'=g'_1 \dots
g'_j$, where the $g'_i$ are either from $\Sigma (X', \Q)$, or of the form
$f'_{s,t, \phi}$.  Extend each $g'_i $ to a function $g_i$ in $ T(X, \P)$, by
letting $g_i|_{P_l}$ be the identity.  It is clear that the $g_i$ are
either in $\Sigma(X,\P)$ or are of the form $f_{s,t, \phi}$. Moreover,
as $h'$ is the identity on $P_l$, $h'= g_1 \dots g_j$.
Hence $h'\in  \langle \Sigma(X,\P) \cup U^{\P}\rangle$, and  so
$ f=g'h' \in \langle \Sigma(X,\P) \cup U^{\P}\rangle$, as required.
\end{proof}

\begin{cor}
Let $\P$ be a partition on $X$, such that $\P$ has exactly $m_i \ge 2$ blocks
  of size $n_i \ge 2$, $i=1, \dots, p$, blocks of unique sizes
  $l_1, \dots , l_q$, where $l_i \ge 2$, and $t$ singleton blocks (where
  $p$, $q$, $t$ might be $0$).  Then the rank of $\Sigma(X, \P)$ modulo $S(X,\P)$ is
  $$\binom{p+q}{2}+p+h(p,q,t)$$
where $h(p,q,0)=0$, $h(p,q,1)=p+q$ and $h(p,q,t)=p+q+1$ if $t \ge 2$.
\end{cor}
\begin{proof}
If $t=0$,  we may take $s=p+q$ and $r=p$ in Theorem \ref{th:TXP:SigXP}.

If $t=1$, with $s=p+q+1$ and $r=p$, we get that the rank of $\Sigma(X,\P)$ modulo
$S(X,\P)$ equals
$$\binom{p+q+1}{2}+p=\binom{p+q}{2}+(p+q)+p.$$
Finally, if $t=2$, the result follows analog with $s=p+q+1$, $r=p+1$.
\end{proof}

\section{The rank of $\Sigma(X,\P)$ over $S(X, \P)$}
\label{section-rel-rank-2}

As in the previous sections, suppose $\P=\{P_1,\ldots ,P_n\}$, with
$|P_i|\le|P_{i+1}|$ and let $l_1 < l_2< \dots < l_r$ be the distinct sizes of
blocks in $\P$.

For $i \le r-1$, let $\mathcal{B}_i$ be the set of all  mappings $f \in
\Sigma(X, \P)$ such that there are $P_j$, $P_{j'}$, $P_k$, $P_{k'}$ with
$|P_j|=l_i=|P_{j'}|$, $|P_k|=l_{i+1}=|P_{k'}|$, such that

\begin{enumerate}
  \item $f$ maps $P_j$ injectively to $P_k$
  \item $f$ maps $P_{k'}$ surjectively onto $P_{j'}$
  \item $f$ maps every other block bijectively to a block of the same size.
\end{enumerate}

We do not exclude the possibility that $j=j'$ or $k=k'$. Clearly,
$\mathcal{B}_i $ is non-empty for all $i \le r-1$, and any element of
$\mathcal{B}_i$ has image size $|X|-l_{i+1}+l_i$.

\begin{lemma}\label{mathcalB}
  If $\genset{S(X, \P), U}=\Sigma(X, \P)$ for some $U\subseteq \Sigma(X, \P)$,
  then $\mathcal{B}_i\cap U\neq \emptyset$ for every $i \le r-1$.
\end{lemma}
\begin{proof}
  Let $i\in \{1,\ldots, r-1\}$ be arbitrary. Then there is
  an $f \in \mathcal{B}_i$ such that $\overline f =(j \ k)$  with $j <k$ and
  $j$ and $k$ are minimal and maximal among those indices of blocks with sizes
  equal to $|P_j|=l_i$ and $|P_k|=l_{i+1}$, respectively. By assumption, there exist
  $g_{1},\ldots, g_{m} \in S(X, \P)\cup U$ such that $f=g_{1}\ldots g_{m}$ and
  hence $\overline{f}=\overline{g_1}\ldots \overline{g_m}$. Since
  $j\overline{f}=k$, it follows that
  $(j)\overline{g_{1}}\ldots\overline{g_{m}}=k$. It follows, since $f|_{P_j}$ is
  injective, and by the minimality of $j$, that $(j)\overline{g_1},
  (j)\overline{g_1}\overline{g_2},\ldots , (j)\overline{g_1}\overline{g_2}\cdots
  \overline{g_m}\geq j$.

  Let $u$ be the least value for which $|P_j|<|P_{(j)\overline{g_1}\cdots
  \overline{g_u}}|$, and let $t=(j)\overline{g_1}\cdots
  \overline{g_{u-1}}$.
  Then at least
  $|P_{(t)\overline{g_u}}|-|P_{t}|=|P_{(t)\overline{g_u}}|-|P_{j}|$ elements of
  $P_{(t)\overline{g_u}}$ are not in the image of $g_{u}$. But $f$ has image size
  $|X|-|P_k|+|P_j|$ and so $g_u|_{P_{t}}$ is injective,
  $|P_{(t)\overline{g_u}}|=|P_k|$, and $X\setminus P_{(t)\overline{g_u}}$ is
  contained in the image of $g_u$.  It follows that
  $\overline{g_u}$ is a permutation that maps every block other than
  $P_{t}$ onto a block of equal or smaller size.

  Since $\overline{g_u}$ is a permutation, and there is exactly one block of
  $\P$ mapped to a larger block, there is also exactly one block $P_{k'}$
  which is mapped to a smaller block $P_{j'}$.  Due to the restriction on the
  size of the image of $f$, it follows that
  $|P_{(t)\overline{g_{u}}}|=|P_k|$ and $|P_{j'}|=|P_j|$.
  As $g_u$ maps every block other than $P_{t}$ surjectively onto its image
  block, it follows that $g_u \in \mathcal{B}_i\cap U$.
\end{proof}


For each $i \le r$ let $\mathcal{C}_i$ be the set of all $f \in \Sigma (X, \P)$ such that
\begin{enumerate}
  \item $f$ maps each block to one of the same size (potentially itself);
  \item there is one block of size $l_i$ whose image under $f$ has size $l_i-1$;
  \item $f$ maps all other blocks injectively.
\end{enumerate}
Clearly, any such $f$ has images size $|X|-1$, and $\mathcal{C}_i$ is non-empty
except when $i=1$ and $l_1 = 1$.

\begin{lemma} \label{mathcalC}
  If $\genset{S(X, \P), U}=\Sigma(X, \P)$ for some $U\subseteq \Sigma(X, \P)$,
  and $i\in \{1, \ldots, r\}$ is such that either $i=1$ and $l_1\ne 1$ or $i\ge
  2$ and $l_i-l_{i-1}\ge 2$, then $\mathcal{C}_i\cap U\neq \emptyset$.
\end{lemma}
\begin{proof}
  Let $f\in \mathcal{C}_i$ be arbitrary.  Then $f=h_1 h_2 \dots h_m$ for some
  $h_1, h_2, \ldots, h_m \in S(X, \P)\cup U$. As mentioned above, the image of
  $f$ has size $|X|-1$.

  Let $z$ be the smallest index for which there is a block $P_k$ of size $l_i$
  that is not contained in the image of $h_1 \cdots h_z$.  Clearly, the the
  image of $h_z$ must contain $l_i-1$ elements of $P_k$.  Since $h_z\in
  \Sigma(X, \P)$, it follows that $\overline{h_z}$ is a permutation.  We set
  $j=(k)\overline h_z^{-1}$.

  We will show that $|P_j|=l_i$. By way of contradiction, assume that $|P_j|
  < l_i$. This is not possible for $i=1$, and if $i \ge 2$ then our condition on
  $i$ implies that $|P_j| < l_i-2$. However in the latter case, there would be at
  least two elements of $P_k$ that were not in the image of $h_z$, contradicting
  the assumption that $f$ has image size $|X|-1$. So $|P_j| \ge l_j$.

  If $|P_j| > l_i$ then (as $\overline h_z$ is a permutation on a finite set)
  there must be one other index $j'$ such that $j' \overline h_z =k'$ with
  $|P_{j'}|<|P_{k'}|$ and $k \ne k'$.  However, in this case  $P_{k'}$  and
  $P_k$ would not be contained in the image of $h_z$, once again contradicting
  the assumption on the image size of $f$.

  We have shown that $|P_j|=l_i$. Note that $h_z$ must map $X\setminus P_j$
  bijectively to $X\setminus P_k$, once again by considering the size of the
  image of $f$.  It follows that $h_z \in \mathcal{C}_i\cap U$.
\end{proof}

We define
$$\mathcal{B}=\bigcup_{i=1}^{r-1} \mathcal{B}_i\qquad\text{and}\qquad
\mathcal{C}=\bigcup_{i=1}^{r}\mathcal{C}_i.$$

Let $f\in S_n$. Then $g \in \Sigma(X,\P)$ is said to be a \emph{companion of
$f$} if
\begin{enumerate}
  \item $\overline{g}=f$;
  \item if $|P_i|\leq |P_{(i)f}|$, then $g|_{P_i}:P_i\to P_{(i)f}$ is
    injective;
  \item if $|P_i|\geq |P_{(i)f}|$, then
    $g|_{P_i}:P_i\to P_{(i)f}$ is surjective.
\end{enumerate}

\begin{lemma}\label{lemma-B-C}
  If $f\in\Sigma(X,\P)$ and there is a companion for
  $\overline{f}$ in $\langle S(X,\P), \mathcal{B}\rangle$, then $f \in
  \langle S(X,\P), \mathcal{B}, \mathcal{C}\rangle$.
\end{lemma}
\begin{proof}
  If $k\in \{1,\ldots, n\}$ is such that $|P_k|=l_i>1$ for some $i$, then there
  exists $t_k\in \mathcal{C}_i$ which is the identity outside $P_k$.  It is
  well-known that for any finite set $Y$ with at least two elements, every
  function on $Y$ is a product of permutations and a fixed function with image
  size $|Y|-1$. Therefore $t_k$ and $S(X, \P)$ generate every element of $T(X,
  \P)$ which maps $P_k$ to $P_k$ and fixes $X\setminus P_k$.  It follows that
  every $f\in \Sigma(X, \P)$ such that $\overline{f}$ is the identity belongs to
  $\langle S(X,\P), \mathcal{C} \rangle$.

  Let $f \in \Sigma(X,\P)$. Then by assumption there exists $g \in \langle
  S(X,\P),\mathcal{B}\rangle$ such that $g$ is a companion for $\overline{f}$.
  From the preceding paragraph, there is an idempotent $e\in \langle S(X, \P),
  \mathcal{C}\rangle$ such that $\ker(e)=\ker(f)$. It follows that there exists
  $h\in S(X, \P)$ such that $f=ehg\in \langle S(X, \P), \mathcal{B},
  \mathcal{C}\rangle$.
\end{proof}

\begin{lemma}\label{lemma-B}
  If $f \in \Sigma(X,\P)$, then there exists a companion for $\overline{f}$ in
  $\langle S(X,\P), \mathcal{B} \rangle$.
\end{lemma}
\begin{proof}
  Since every permutation is a product of disjoint cycles,  there exists a
  companion for $f \in \Sigma(X,\P)$ in $\genset{S(X, \P), \mathcal{B}}$ if and
  only if there is a companion in $\langle S(X,\P), \mathcal{B} \rangle$ for
  every cycle in $S_n$.

  For any $k \le n-1$ let $f_{(k\ k+1)}$ be such that $\overline{f_{(k\
  k+1)}}=(k \ k+1)$, $f_{(k\ k+1)}|_{P_{k+1}}$ maps onto $P_k$, the image of
  $f_{(k\ k+1)}|_{P_k}$ is a section for the kernel of $f_{(k\
  k+1)}|_{P_{k+1}}$, and $f_{(k\ k+1)}$ is the identity outside of $P_k \cup
  P_{k+1}$.  Since $|P_{k+1}f|=|P_k|$, it follows that $f$ is injective on $P_k$
  and so $f_{(k \ k+1)}$ is a companion for $(k \ k+1)$.  Moreover, $f_{(k\
  k+1)}$ belongs to $\mathcal{B}$ when $|P_k|<|P_{k+1}|$ and it belongs to
  $S(X, \P)$ when $|P_k|=|P_{k+1}|$.

  Suppose that $i<j$. Then it is straightforward to check that
  $$f_{(i \ j)}=f_{(j-1 \ j)} f_{(j-2 \ j-1)} \cdots f_{(i+1 \ i+2)} f_{(i \
  i+1)} f_{(i+1 \ i+2)} f_{(i+2 \ i+3)} \cdots f_{(j-2 \ j-1)} f_{(j-1 \ j)}$$
  is a companion for $(i\ j)$.

  We proceed by induction on the length $k$ of a cycle.  Suppose that for
  some $k$ with $2\leq k<n$, there exists a companion in $\langle S(X,\P),
  \mathcal{B} \rangle$ for every cycle of length at most $k$.
  Let $h=(x_1\ldots x_{k+1})$ and let $x_j=\min\{x_1,\ldots , x_{k+1}\}$. Then
  $$h=(x_{j+1}x_{j+2}\ldots x_{k+1}x_1\ldots x_{j-1})(x_j\ x_{j+1}).$$
  By induction, there is a companion $h_1\in \langle S(X,\P), \mathcal{B}
  \rangle$ for $(x_{j+1}x_{j+2}\ldots
  x_{k+1}x_1\ldots x_{j-1})$.  It follows that $|\im(h_1|_{P_{x_{j-1}}})| =
  \min\{|P_{x_{j-1}}|,|P_{x_{j+1}}|\} \ge |P_{x_{j}}|$ by the minimality of
  $x_j$.

  Let $g \in S(X,\P)$ be such that $g$ maps a subset of
  $\im(h_1|_{P_{x_{j-1}}})\subseteq P_{x_{j+1}}$ onto a section of the
  kernel of $f_{(x_j \ x_{j+1})}|_{P_{x_{j+1}}}$, and is the identity outside of
  $P_{x_{j+1}}$. Since $f_{(x_j \ x_{j+1})}|_{P_{x_{j+1}}}$ has $|P_{x_j}|$ kernel
  classes such $g$ exists due to our estimate above.
  It follows that $h_1g f_{({x_j}\ {x_{j+1}})}$ is a companion
  for $h$.
\end{proof}

The two previous results imply the following corollary.

\begin{cor}\label{cor:SigXPgen}
  $\Sigma(X,\P)=\langle S(X, \P), \mathcal{B}, \mathcal{C} \rangle$.
\end{cor}

\begin{theorem}\label{th:SigXP:SXP}
  Let $U$ be a set that contains one representative from each $\mathcal{B}_i$,
  for $i \le r-1$, and one representative from each $\mathcal{C}_i$, for all $i$
  that satisfy the condition in the statement of Lemma \ref{mathcalC} (i.e.
  either $i=1$ and $l_1 \ge 2$ or $i\ge 2$ and $l_i-l_{i-1}\ge 2$). Then
  $\Sigma(X, \P)=\langle S(X, \P), U \rangle$.
\end{theorem}
\begin{proof}
  By Corollary \ref{th:SigXP:SXP}, it suffices to show that $S(X,
  \P)\cup U$ generates $\mathcal{B}$ and $\mathcal{C}$.

  Considering $\mathcal{C}$, we will first show that for each $P_x$ with $|P_x|
  > 1$, there exists an $f_x$ such that $P_y f_x \subseteq P_y$ for all $y \le
  n$, $f_x$ has image size $n-1$, and that $P_x$ is the unique block that is not
  mapped injectively by $f_x$.

  Let $i$ be such that $l_i=|P_x|$. If either $i=1$ (in which case $l_1
  =|P_x|\ge 2$), or $l_i-l_{i-1}\ge 2$, then there exists $f_{x'}
  \in U \cap \mathcal{C}_i$, that is not injective on $P_{x'}$ with
  $|P_{x'}|=|P_x|$. But then $f_x \in S(X,\P) f_{x'} S(X,\P)$ by Lemma
  \ref{lem:Gaction}.

  If $i \ge 2 $ and $l_i-l_{i-1}=1$, let $h_{i-1}$ be the element in
  $\mathcal{B}_{i-1}\cap U$. Let $f_{xy}$ be a mapping that maps  $P_x$ onto some $P_y$ with
$|P_y|=l_{i-1}=|P_x|-1$, maps $P_y$ injectively to $P_x$ and is the identity everywhere else.
  By Lemma \ref{lem:Gaction}, we have that $f_{xy} \in S(X, \P)h_{i-1}S(X, \P)$.
 Now $f_x:=f_{xy}^2$ can easily be checked to have the claimed properties.

Now for general $h \in \mathcal{C}_i$,  $l_i \ne 1$, as otherwise $C_i$ would be empty.
Choose a $P_x$ with $|P_x|=l_i$, then
$g  \in S(X,\P) f_x S(X,\P)$ by Lemma \ref{lem:Gaction}.
  Hence $\mathcal{C} \subseteq \langle S(X,\P), U \rangle$, as required.

Now, for each $z$ with $|P_z|=l_i  \ge 2$,
there is a function $ f_z \in \langle\mathcal{C} \subseteq S(X,\P) \cup U\rangle$
that maps $P_z$ to itself, is the identity everywhere
else,  and has an image that intersects $P_z$ with size $l_i-1$. Consider the subsemigroup $Q_z$ of
$T(X, \P)$, consisting of all elements that map $P_z$ into itself and are the identity outside of
$P_z$. $Q_z$ is clearly isomorphic to $T_{P_z}$, the full transformation semigroup on $P_z$.
This semigroup is generated
by its units together with a transformation of rank $|P_z|-1$. It follows that
$Q_z \subseteq \langle (S(X,\P)\cap S_z) \cup \{f_z\} \rangle$
for every $z$ (note that  this also holds trivially
if $|P_z|=1$).

Now consider any element $h \in \mathcal{B}_i$, for $i \le r-1$, and let $P_z$ be the unique part of $\P$
that is not mapped injectively by $h$.
 There exists an $h' \in Q_z$ that has the same kernel classes on $P_z$ as $h$. But both
$h \in \mathcal{B}_i$ and $h' \in Q_z$ only have singleton kernel classes outside of $P_z$, and hence
$h \in S(X, \P) h' S(X,\P)$ by Lemma \ref{lem:Gaction}. So $h \in \langle S(X,\P) \cup U\rangle$, as required.

\end{proof}

\begin{cor} Let $\P$ be a partition on $X$, such that $\P$ has exactly $m_i \ge 2$ blocks
  of size $n_i \ge 2$, $i=1, \dots, p$, blocks of unique sizes
  $l_1, \dots , l_q$, where $l_i \ge 2$, and $t$ singleton blocks (where
  $p$, $q$, $t$ might be $0$).  Then
$$\rank (\Sigma(X, \P):S(X,\P))= p+q+g'(t)-1+l$$
  where
\begin{itemize}
\item $g'(0)=0$ and $g'(t)=1$ for $t \ge 1$,
\item $l$ is the number of values $s$ for which $\P$ has a block of size $s\ge  2$, but no block of size $s-1$.
\end{itemize}
\end{cor}
\begin{proof}
From Lemma \ref{mathcalB}, Lemma \ref{mathcalC}, and Theorem \ref{th:SigXP:SXP} it follows that
$\rank(\Sigma(X,\P):S(X,\P))$ is one less than the number of distinct block sizes of $\P$ plus
the number of  block sizes that satisfy the conditions mentioned in Lemma \ref{mathcalC}. The first of these
numbers is $p+q+g'(t)-1$ and the second is $l$.
\end{proof}

\section{Problems}

Let $X$ be a finite set, let $\P$ be a partition and $S$ be a section, that is, for every $P\in \P$ we have  that $S\cap P$ is a singleton set. Given a set $Y\subseteq X$, we say that $f\in T(X)$ stabilizes $Y$ if $Yf\subseteq Y$.  Now consider the semigroup
\[
T(X,\P,S)=\{f\in T(X)\mid  \mbox{$f$ stabilizes $\P$ and $S$}\} .
\]
This semigroup, in addition to the obvious similarities with $T(X,\P)$, has many interesting properties:
\begin{enumerate}
\item both $T(X)$ and $PT(X)$, the semigroup of partial transformations on $X$, are examples of semigroups of this type; for instance,  $T(X)$ is $T(X,\{\{x\}\mid x\in X\},X)$ and $PT(X)$ is isomorphic  (for an element $0\not\in X$)  to $T(X\cup \{0\},\{X\cup\{0\}\},\{0\})$ (see \cite{AK0,AK1}).
\item Let $e^2=e\in T(X)$; the centralizer of $e$ in $T(X)$ is $C(e)=\{f\in T(X) \mid fe=ef\}$. Then  $C(e)=T(X,\ker(e),Xe)$ (see \cite{AK0,AK1}). In this setting, $T(X)$ is the centralizer of the identity and $PT(X)$ is the centralizer of a constant map.
\item $T(X,\P,S)$ is regular if and only if either
\begin{enumerate}
\item no part in $\P$ has more than $2$ elements; or
\item at most one of the parts in $P$ has size larger than $1$.
\end{enumerate}
(See \cite{AK1}.)
\item The singular elements of a regular $C(e)$ are generated by idempotents if and only if $e$ is the identity or a constant (see \cite{AAK}).
\item Taking into account that in the Cayley table of a semigroup, each column (seen as map) is contained in the centralizer of each row, the semigroups $T(X,\P,S)$ have some surprising consequences in equational logic (see \cite{AK-1}).
\end{enumerate}

Therefore given the importance of $T(X,\P,S)$ and its similarities with the semigroups $T(X,\P)$, the following problems are very natural.

\begin{prob}\label{p1}
Find the rank $T(X,\P,S)$, when $\P$ is uniform and $T(X,\P,S)$ is regular. (Given the results above, this amounts to find the rank of $T(X,\P,S)$ when all parts in $\P$ have exactly two elements.)
\end{prob}

The previous problem is a partial analogous of the main result in \cite{as}.
The full analogous would be the following.

\begin{prob}
Find the rank $T(X,\P,S)$, when $\P$ is uniform.
\end{prob}

Given the importance of regular semigroups in semigroup theory the following problem is also natural.
\begin{prob}
Find the rank of the regular semigroups $T(X,\P,S)$. (The solution of this problem requires the solution of Problem \ref{p1}.)
\end{prob}

Obviously, the ultimate goal of this sequence of problems would be the solution of the problem analogous to the main problem solved in this paper.

\begin{prob}
Find the rank of the regular semigroups $T(X,\P,S)$.
\end{prob}

\end{document}